\documentclass[11pt,reqno]{article}
\usepackage[margin=1in,letterpaper]{geometry}
\usepackage{microtype}
\usepackage{amssymb}
\usepackage{amsmath} 
\usepackage{mathtools}
\usepackage{tikz-cd}
\usepackage{amsthm}
\usepackage{thmtools, thm-restate}
\usepackage{accents}
\usepackage{authblk} 
\usepackage{hyperref} 
\hypersetup{
    colorlinks=true,
    linkcolor=blue,  
    urlcolor=blue,
    citecolor=blue
}
\usepackage{enumitem} 
\usepackage{graphicx} 
\usepackage{cleveref}
\let\cref\Cref

\usepackage[
    backend=biber,
    style=alphabetic,
    maxnames = 99,
    url=false,      
    doi=true,      
    isbn=false,    
    eprint=false    
]{biblatex}

\newtheorem{theorem}{Theorem}[section]
\newtheorem{lemma}[theorem]{Lemma}

\newtheorem{proposition}[theorem]{Proposition}
\newtheorem{corollary}[theorem]{Corollary}
\theoremstyle{definition}
\newtheorem{definition}[theorem]{Definition}
\newtheorem{example}[theorem]{Example}

\newtheorem{remark}[theorem]{Remark}

\definecolor{darkblue}{rgb}{0.0, 0.0, 0.8}
\definecolor{darkred}{rgb}{0.8, 0.0, 0.0}
\definecolor{darkgreen}{rgb}{0.0, 0.8, 0.0}

\newcommand{\rmH}{\mathrm{H}}

\newcommand{\rmQ}{\mathrm{Q}}

\newcommand{\rmS}{\mathrm{S}}

\newcommand{\bbQ}{\mathbb{Q}}
\newcommand{\bbR}{\mathbb{R}}
\newcommand{\bbS}{\mathbb{S}}

\newcommand{\bbZ}{\mathbb{Z}}

\newcommand{\SSet}{\mathbf{S S e t}}
\newcommand{\topo}{\mathbf{T  o p}}
\newcommand{\Ho}{\mathbf{H o}}
\newcommand{\Vect}{\mathbf{V e c}}

\newcommand{\dgh}{d_\mathrm{GH}}
\newcommand{\dhi}{d_\mathrm{HI}}
\newcommand{\di}{d_\mathrm{I}}
\newcommand{\dhaus}{d_\mathrm{H}}

\newcommand{\VR}{\operatorname{VR}}
\newcommand{\im}{\operatorname{im}}
\newcommand{\id}{\operatorname{id}}
\newcommand{\Hom}{\operatorname{Hom}}

\newcommand{\pcHVR}[3][\epsilon]{\rmH^{#2} \left(\VR_{#1} \left(#3\right)\right)}
\newcommand{\op}{\operatorname{op}}

\usepackage[margin=1in]{geometry}
\usepackage{setspace}
\usepackage{amsmath, amssymb, amsthm, mathtools}
\usepackage{enumitem}
\setlist[itemize]{topsep=3pt, itemsep=3pt, leftmargin=*}
\usepackage{newtxtext,newtxmath}
\usepackage{tikz-cd}

\AtBeginDocument{%
  \abovedisplayskip=6pt plus 3pt minus 3pt
  \belowdisplayskip=6pt plus 3pt minus 3pt
}

\DeclareBibliographyDriver{article}{%
  \usebibmacro{bibindex}%
  \usebibmacro{begentry}%
  \printnames{author}%
  \setunit{\labelnamepunct}\newblock
  \printfield{title}%
  \newunit\newblock
  \printfield{journaltitle}%
  \setunit*{\addspace}%
  \printfield{volume}%
  \setunit{\addcolon\addspace}%
  \printfield{number}%
  \setunit{\addcomma\addspace}%
  \printfield{year}%
  \setunit{\addcomma\addspace}%
  \printfield{pages}%
  \newunit\newblock
  \iffieldundef{doi}{
    \printfield{url}
  }{
    \printfield{doi}
  }
  \newunit\newblock
  \usebibmacro{finentry}
}

\bibliography{usualrefs}
\bibliography{refs} 

\newcommand{\APL}{\operatorname{A_{PL}}}
\newcommand{\diam}{\operatorname{diam}}
\newcommand{\catC}{\mathcal{C}}
\newcommand{\catD}{\mathcal{D}}

\newcommand{\CDGA}{\mathbf{C D G A}}

\newcommand{\vs}{\Vect}
\newcommand{\dgvs}{\mathbf{D G V e c}}
\newcommand{\GVect}{\mathbf{G V e c}}

\newcommand{\smm}[1]{\mathfrak{M}(#1)}
\newcommand{\pHVR}[3][ t ]{\rmH_{#2}(\VR_{#1} (#3))}

\newcommand{\Afunc}{A_\bullet}
\newcommand{\Bfunc}{B_\bullet}

\newcommand{\Xfunc}{X_\bullet}
\newcommand{\Yfunc}{Y_\bullet}

\newcommand{\Gfunc}{G_\bullet}

\newcommand{\Ufunc}{U_\bullet}
\newcommand{\Vfunc}{V_\bullet}
\newcommand{\Wfunc}{W_\bullet}

\usepackage{graphicx} 

\title{Persistent Sullivan Minimal Models of Metric Spaces}
\author{Ling~Zhou}
\affil{Department of Mathematics, 
	   Duke University.
		\thanks{\texttt{ling.zhou@duke.edu}}}
\date{\today}

\begin{document}
\maketitle

\begin{abstract} 
    We extend classical tools from rational homotopy theory to topological data analysis by introducing persistent Sullivan minimal models of persistent topological spaces. Our main result establishes that the interleaving distance between such models in the homotopy category of CDGAs is stable with respect to the homotopy interleaving distance of the underlying spaces. For Vietoris–Rips filtrations of metric spaces, this yields new persistent invariants that are more discriminative than persistent homology. We further show that these models provide sharper lower bounds for the Gromov–Hausdorff distance than those obtained from persistent homology or persistent rational homotopy groups. 
\end{abstract}

\section{Introduction}

In topological data analysis, \emph{persistent homology} is widely used to study the evolution of homological features across a filtration of spaces \cite{frosini1990distance,frosini1992measuring,robins1999towards,zomorodian2005computing,cohen2007stability,edelsbrunner2008persistent,carlsson2009topology,carlsson2020persistent}. 
Its foundational stability theorem states that the interleaving distance $\di$ between the persistent homology of two metric spaces is bounded above by a constant multiple of the Gromov–Hausdorff distance $\dgh$ between the underlying metric spaces:

\begin{restatable}[{Stability theorem for persistent homology, {\cite{chazal2009gromov,CSO14}}}]{theorem}{dbstab}\label{thm:dgh-classical} 
Let $X$ and $Y$ be two totally bounded metric spaces, and let \(\VR_\bullet(X)\) and \(\VR_\bullet(Y)\) denote their respective Vietoris-Rips filtrations. 
Then, for any degree $k{\geq 0},$
\[\di (\pHVR[\bullet]{k}{X},\pHVR[\bullet]{k}{Y})\leq 2\cdot \dgh(X,Y).\]
\end{restatable}
 
This result provides a robust theoretical justification for using persistent homology in data analysis.
However, persistent homology alone has limited discriminative power, as homology groups do not fully determine the topological structure of a space. 
This limitation motivates the search for more informative persistent invariants.

In this work, we investigate persistent invariants derived from Sullivan minimal models, extending classical tools from rational homotopy theory to the persistence setting. 
Our focus is on defining persistent Sullivan minimal models for persistent topological spaces and establishing stability results in the homotopy category of commutative differential graded algebras (CDGAs).

The theory of minimal models was initiated by Quillen in 1969. In \cite{Quillen1969}, he established an equivalence between the homotopy category of simply-connected rational spaces of finite type and the homotopy category of differential graded Lie algebras of finite type.
Inspired by the differential forms on a manifold, Sullivan constructs a morphism from the homotopy category of topological spaces to the homotopy category of (rational) CDGAs \cite{sullivan1977infinitesimal}. 
When restricted to the homotopy category of simply-connected rational spaces with finite Betti numbers and the corresponding category of CDGAs, this morphism becomes an equivalence.
Thus, Sullivan minimal models offer an algebraic characterization of the rational homotopy type of simply-connected spaces.

\paragraph*{Related work.} The idea of studying rational homotopy invariants in persistent settings has been considered in previous works. Peterson and Pancia \cite{petersonwebsite} developed a Mathematica package to compute Sullivan minimal models of Vietoris–Rips filtrations, aiming to facilitate the study of persistent rational homotopy groups. Independently, Hess, Lavenir, and Maggs \cite{hess2023cell} developed a model-categorical framework for tame persistent CDGAs and constructed explicit minimal models using interval-based Hirsch extensions. While their work emphasizes algebraic constructions and rational Postnikov towers, our approach introduces persistent minimal models from a geometric and topological perspective and is the first to establish metric stability results for such models.

Other developments in persistent homotopy theory include the work of Mémoli and the author \cite{memoli2024persistent}, which studies properties of persistent homotopy groups in both the rational and non-rational settings. 
Related work on non-rational persistent homotopy theory also appears in \cite{blumberg2017universality, lanari2023rectification, Jardine2019DataAH, jardine2020persistent}.

\paragraph*{Contributions.} 

We extend Sullivan’s framework to the persistent setting by introducing the notion of a \emph{persistent Sullivan minimal model}. A \emph{persistent space} is a functor \(\Xfunc\)  from the poset category \( (\mathbb{R}, \leq) \) to the category \(\topo\) of path-connected topological spaces of \emph{finite type}—that is, spaces whose homotopy groups are finite-dimensional \(\mathbb{Q}\)-vector spaces in each degree and vanish above some degree.  
A persistent Sullivan minimal model of $\Xfunc$ consists of a persistent minimal Sullivan algebra \(\smm{X_\bullet}\), together with a family of quasi-isomorphisms \(\mu_{X_\bullet}\), each relating \(\smm{X_t}\) to the CDGA \(\APL(X_t)\) of polynomial differential forms on \(X_t\).  
These maps are compatible with the structure maps of the persistence module, up to homotopy; see \cref{def:psmm}. 

Throughout, all vector spaces are assumed to be defined over \(\mathbb{Q}\) throughout.
Let \(\CDGA\) and \(\GVect\) denote the categories of CDGAs and graded vector spaces, respectively.

We establish several \emph{stability} results concerning the interleaving distance $\di$ between persistent Sullivan minimal models in the homotopy category of CDGAs, denoted \(\Ho(\CDGA)\). 
These results demonstrate that the interleaving distance in \(\Ho(\CDGA)\) is controlled by the homotopy interleaving distance $\dhi$ between the underlying persistent spaces, which is a homotopy-invariant refinement of the ordinary interleaving distance introduced by Blumberg and Lesnick~\cite{blumberg2017universality}.
See \cref{sub:stability} for further details, where we employ techniques similar to those in \cite{memoli2023persistent,memoli2024persistent} that studied the stability results for persistent homotopy groups (including persistent rational homotopy groups) and persistent cohomology rings, respectively.
See also \cite{zhou2023} for more details.

\begin{restatable}
{theorem}{stabofPM}\label{thm:stab-PM} 
Let $\Xfunc$ and $\Yfunc$ be two persistent spaces. Assume that $\smm{\Xfunc}$ and $\smm{\Yfunc}$ are persistent Sullivan minimal models of $\Xfunc$ and $\Yfunc$, respectively. Then 
\[\di^{\Ho(\CDGA)}(\smm{\Xfunc}, \smm{\Yfunc}) \leq \dhi(\Xfunc,\Yfunc).\]
For totally bounded metric spaces \(X\) and \(Y\), we have
\[\di^{\Ho(\CDGA)}( \smm{\VR_\bullet(X)},\smm{\VR_\bullet(Y)} )\leq 2\cdot\dgh(X,Y).\]
\end{restatable}

Furthermore, we show that persistent Sullivan minimal models yield \emph{sharper lower bounds} for the Gromov--Hausdorff distance between metric spaces than either persistent homology or rational homotopy groups. In particular, the interleaving distance between persistent Sullivan minimal models is bounded below by that of both persistent cohomology and persistent rational homotopy groups.
 
We denote by \(\rmH: \CDGA \to \GVect\) the homology functor. 
To indicate the grading, we use \(\rmH^*\) for cochain algebras and \(\rmH_*\) for chain algebras when necessary. 
However, when there is no risk of confusion, we may omit these superscripts and subscripts.

\begin{restatable}{theorem}{lowerBound}\label{thm:lower-bound} \begin{enumerate}
    \item \label{thm:interleaving-H}
    Let $\Afunc$ and $\Bfunc$ be two persistent CDGAs. Then,
    \begin{equation}\label{eq:stab 1}
        \di^{\GVect}(\rmH\circ \Afunc, \rmH\circ \Bfunc)\leq  \di^{\Ho(\CDGA)}(\Afunc,\Bfunc).
    \end{equation}
\item \label{thm:interleaving-V,W} Let $\wedge\Vfunc$ and $\wedge \Wfunc$ be two simply-connected persistent minimal Sullivan algebras. Then
    \begin{equation}\label{eq:stab 2}
        \di^{\GVect}(\Vfunc,\Wfunc)\leq  \di^{\Ho(\CDGA)}(\wedge\Vfunc, \wedge \Wfunc).
    \end{equation}
\end{enumerate}
\end{restatable}

These results follow from functorial properties of minimal models and their indecomposables. Examples show that the inequalities can be strict (\cref{subsubsec:lower bounds}), illustrating the enhanced {discriminative power} of persistent Sullivan minimal models compared to traditional persistent homology.

\paragraph*{Organization.} 
In \cref{sec:preliminary},
We provide an overview of persistence theory (\cref{sec:persistence}) and Sullivan minimal models (\cref{sec:smm background}). 
In \cref{sec:pSMM}, we define a persistent Sullivan minimal models of persistent topological spaces and establish their properties.
Our first main result appear in \cref{thm:stab-PM}, showing that the interleaving distance in the homotopy category of CDGAs is stable under the Gromov-Hausdorff distance. 
Finally, we demonstrate in \cref{thm:lower-bound} that persistent Sullivan minimal models provide stronger distinguishing power than persistent (co)homology and rational homotopy groups.

\paragraph*{Acknowledgements.} 
The author was partially supported by the NSF through grants RI-1901360, CCF-1740761, CCF-1526513, and DMS-1723003. The author thanks Dr.~Eric Peterson for valuable discussions.
The author also thanks Prof.~Facundo M{\'e}moli for engaging in insightful discussions about the persistent rendition of Sullivan models and the pursuit of lower bounds for the interleaving distance between persistent Sullivan minimal models of metric spaces.

\section{Preliminary} 
\label{sec:preliminary}

In \cref{sec:persistence}, we provide background on persistence theory, covering persistent objects, the interleaving distance, the Vietoris–Rips filtration, the Gromov–Hausdorff distance, and the stability of persistent homology.  
In \cref{sec:smm background}, we review the definition and basic properties of Sullivan minimal models. 

\subsection{Persistence theory}
\label{sec:persistence}

\begin{definition}
    A \emph{persistent object} in a small category \(\catC\) is a functor \(F_\bullet: (\bbR, \leq) \to \catC\), where \((\bbR, \leq)\) is the poset category of real numbers under the usual order. More explicitly, a persistent object \(F_\bullet\) consists of:  
    \begin{itemize}
        \item an object \(F_t\) of \(\catC\) for each \(t \in \bbR\),
        \item a morphism \(f_{t,s}: F_t \to F_s\) for each \(t \leq s\), satisfying:
        \begin{itemize}
            \item \(f_{t,t} = \id_{F_t}\),
            \item \(f_{s,r} \circ f_{t,s} = f_{t,r}\) for all \(t \leq s \leq r\).
        \end{itemize}
    \end{itemize}
\end{definition}

Persistent objects in \(\topo\) are called \emph{persistent topological spaces}, while those in \(\Vect\) are known as \emph{persistent vector spaces} or \emph{(standard) persistence modules}.

\begin{definition}\label{def:interval-like persistent object}
    Let \(\catC\) be a category with a zero object \(0\). For an interval \(I \subset \bbR\) and an object \(M \in \catC\), the \emph{interval-like persistent object} associated to \(I\) and \(M\), denoted \(M[I]\), is the persistent object \(M[I]: (\bbR, \leq) \to \catC\) defined by
    \[
    M[I](t) := 
    \begin{cases}
    M, & t \in I, \\
    0, & t \notin I,
    \end{cases}
    \quad \text{and} \quad
    M[I](t \leq s) := 
    \begin{cases}
    \id_M, & [t,s] \subseteq I, \\
    0, & \text{otherwise.}
    \end{cases}
    \]
\end{definition}

\noindent
For example, in the category \(\Vect\), the interval module \(\bbQ[I]\) is the interval-like persistent object associated to $I$ and \(M = \bbQ\).

\paragraph*{Interleaving distance.}
\label{subsec:di dhi}
Let $ F_\bullet$ and $ F_\bullet'$ be two persistent objects in $\catC$. A natural transformation $f: F_\bullet\Rightarrow  F_\bullet',$
also called a \emph{homomorphism} from $ F_\bullet$ to $ F_\bullet'$, is a family of morphisms in $\catC$: $\{\varphi_t:F_t\to F_t'\}_{t\in \bbR }$ such that for any $t\leq s$ the following diagram commutes:
	\begin{center}
		\begin{tikzcd} 
			F_t \ar[d, "\varphi_t" left] 
			\ar[r,"f_{t,s}"]
			& 
			F_s			
			\ar[d,"\varphi_s "]
			\\
			F_t'
			\ar[r,"f'_{t,s}" below]
			& 
			 F'_{s}.
		\end{tikzcd}
	\end{center} 
    \label{para:Hom(V,W)}  
If $\varphi_t$ is an isomorphism in $\catC$ for each $t\in \bbR $, then $f=\{\varphi_t\}$ is called a \emph{(natural) isomorphism} between $ F_\bullet$ and $ F_\bullet'$, in which case we write $ F_\bullet\cong F_\bullet'.$

\begin{definition}
\label{def:delta-interleaving} 
For any $\delta\geq 0$, a \emph{$\delta$-interleaving} $( F_\bullet, F_\bullet',f,g)$ consists of families of morphisms $\{\varphi_t:F_t\to  F'_{t+\delta}\}_{t\in \bbR }$ and $\{\phi_t:F_t'\to  F_{t+\delta}\}_{t\in \bbR }$ such that the following diagrams commute for all $t\leq s$:
	\begin{center}
		\begin{tikzcd}[column sep={6em,between origins}]
			F_t \ar[dr, "\varphi_t" below left] 	
			\ar[r,"f_{t,s}"]
			& 
			F_s			
			\ar[dr,"\varphi_s " ]&
			\\
			& F'_{t+\delta}
			\ar[r,"f_{t+\delta,s+\delta}'" below]
			& 
			 F'_{s+\delta}
		\end{tikzcd}
\hspace{0.6cm}
	\begin{tikzcd}[column sep={6em,between origins}]
			&  F_{t+\delta}
			\ar[r,"f_{t+\delta,s+\delta}"]
			& 
			 F_{s+\delta}
			\\
			F_t'
			\ar[ur, "\phi_t"] 
			\ar[r,"f'_{t,s}" below]
			& 
			 F'_{s}
			\ar[ur,"\phi_s" below right]&
		\end{tikzcd}
	\end{center} 
and
\begin{center}
		\begin{tikzcd}
			F_t
			\ar[dr, "\varphi_t" below left ] 
			\ar[rr,"f_{t,t+2\delta}"]%
			&& 
			 F_{t+2\delta}		
			\\
			&  F'_{t+\delta}
			\ar[ur,"\phi_{t+\delta}" below right ]
			& 
		\end{tikzcd}
\hspace{0.6cm}
		\begin{tikzcd}
			&  F_{t+\delta}
			\ar[dr,"\varphi_{t+\delta}"]
			& 
			\\
			F_t'
			\ar[ur, "\phi_t"] 
			\ar[rr,"f'_{t,t+2\delta}" below]
			&& 
			 F'_{t+2\delta}.	
		\end{tikzcd}
	\end{center}
\end{definition}

\begin{definition}\label{def:interleaving}
Let $F_\bullet $ and $F_\bullet '$ be persistent objects in $\catC$.
The \emph{interleaving distance} between $ F_\bullet$ and $ F_\bullet'$ is 
		$$\di( F_\bullet, F_\bullet'):=\inf\{\delta\geq0:  F_\bullet \text{ and } F_\bullet'\text{ are }\delta\text{-interleaved}\}.$$
Here we follow the convention that $\inf \emptyset=+\infty.$ We will write $\di^{\catC}$ when we need to highlight the underlying category $\catC$.
\end{definition}

By \cite[Corollary 3.5]{Bubenik2014}, $\di$ is an extended metric on the set of isomorphism classes of persistent objects in $\catC$. In addition, the following property of the interleaving distance plays an important role for the stability results that we will discuss later:

\begin{theorem}[{\cite[\textsection 2.3]{bubenik2015metrics}}]
\label{thm:Bubenik}
For functors $ F_\bullet,\Gfunc:(\bbR,\leq)\to\catC$ and $\rmH:\catC\to \catD$,  
\[ d_{\mathrm{I}}^{\catD}(\rmH F_\bullet,\rmH\Gfunc)\leq d_{\mathrm{I}}^{\catC}( F_\bullet,\Gfunc).\]
\end{theorem}

\paragraph*{Gromov-Hausdorff distance.} The \emph{Hausdorff distance} between two subspaces $X$ and $Y$ of a metric space $Z$ is \label{para:dh}
\[d_\mathrm{H}^Z(X,Y):=\inf\left\{r>0: X\subseteq \bar{B}(Y,r)\text{ and }  Y\subseteq \bar{B}(X,r)\right \}.\]
The \emph{Gromov-Hausdorff distance} between metric spaces $(X,d_X)$ and $(Y,d_Y)$ is the infimum of $r>0$ for which there exist a metric spaces $Z$ and two distance preserving maps $\psi_X:X\to Z$ and $\psi_Y:Y\to Z$ such that $\dhaus^Z(\psi_X(X),\psi_Y(Y))<r$, i.e., 
\begin{equation}\label{eq:dgh}
\dgh(X,Y):=\inf_{Z,\psi_X,\psi_Y}\dhaus^Z(\psi_X(X),\psi_Y(Y)).
\end{equation}
See \cite{edwards1975structure,gromov2007metric} for further details.

\paragraph*{Vietoris-Rips filtration.}
For a metric space $(X,d_X)$ and $ t > 0$, the \emph{Vietoris–Rips complex} $\VR_{ t  }(X)$ is the simplicial complex with vertex set $X$, where 
\begin{center}
    a finite subset $\sigma\subset X$ is a simplex of $\VR_{ t  }(X)$ iff $\diam(\sigma)< t $.
\end{center} 
The collection $\{\VR_{ t  }(X)\}_{ t > 0}$ together with the natural simplicial inclusions forms a simplicial filtration, denoted by $\VR_{\bullet}(X)$, of the power set of $X$. For simplicity, we will use the same notation for both a simplicial complex and its corresponding geometric realization, when there is no danger of confusion.

We remark that the Vietoris–Rips complex can also be defined using the closed condition \(\diam(\sigma) \leq t\) instead. 
All results in this work hold for either convention, and for simplicity, we adopt the open version throughout.

Applying \( k \)-th homology with rational coefficients\footnote{In general, other coefficient fields may be considered, but we restrict to \(\mathbb{Q}\) in this work.} to the filtration \(\VR_{\bullet}(X)\) yields a persistent object, denoted \(\pHVR[\bullet]{ k }{X}\), called the \emph{\(k\)-th persistent homology} of \(X\), given by 
\[
t \mapsto \pHVR[t]{ k }{X},
\]  
together with the maps on homology induced by the natural simplicial inclusions. 

\paragraph*{Homotopy interleaving distance.} 
Blumberg and Lesnick pointed out that the interleaving distance between persistent spaces is not homotopy invariant (see Remark 3.3 of \cite{blumberg2017universality}).
To address this, they introduced \emph{homotopy interleavings}, which serve as homotopy-invariant analogs of interleavings between persistent spaces.

Given two persistent spaces $\Xfunc$ and $\Yfunc$, a natural transformation $f:\Xfunc\Rightarrow \Yfunc$ is an \emph{(objectwise) weak equivalence} if for each $t\in \bbR$, $\varphi_t:X_t\to Y_t$ is a weak homotopy equivalence, i.e., it induces isomorphisms on all homotopy groups. \label{para:weak h.e.} The persistent spaces $\Xfunc$ and $\Yfunc$ are \emph{weakly equivalent}, denoted by $\Xfunc\simeq \Yfunc$, if there exists a persistent space $\Wfunc$ and natural transformations $f:\Wfunc\Rightarrow \Xfunc$ and $g:\Wfunc\Rightarrow \Yfunc$ that are (objectwise) weak equivalences:
\[\Xfunc\xLeftarrow{f} \Wfunc\xRightarrow{g} \Yfunc.\]
The relation $\Xfunc\simeq \Yfunc$ is an equivalence relation (see \cite{blumberg2017universality} for details). For $\delta\geq 0$, two persistent spaces $\Xfunc$ and $\Yfunc$ are \emph{$\delta$-homotopy-interleaved} if there exist persistent spaces $\Xfunc'\simeq \Xfunc$ and $\Yfunc'\simeq \Yfunc$ such that $\Xfunc'$ and $\Yfunc'$ are $\delta$-interleaved, as in \cref{def:delta-interleaving} with $\catC=\topo$. \label{para:weakly equivalent}

\begin{definition}[
{\cite[Definition 3.6]{blumberg2017universality}}] \label{def:HI} The \emph{homotopy interleaving distance} between two persistent spaces $\Xfunc$ and $\Yfunc$ is given by
	$$\dhi(\Xfunc,\Yfunc):=\inf \left\{\delta\geq 0:\Xfunc,\Yfunc \text{ are }\delta\text{-homotopy-interleaved}\right\}.$$
In other words, $d_{\mathrm{HI}}(\Xfunc,\Yfunc)=\inf\left\{\di^{\topo}(\Xfunc',\Yfunc')\hspace{0.5em}\mid\Xfunc'\simeq \Xfunc\text{ and }\Yfunc'\simeq \Yfunc\right\}.$
\end{definition}

Blumberg and Lesnick \cite{blumberg2017universality} strengthened the previously mentioned stability result, \cref{thm:dgh-classical}, using the homotopy-interleaving distance and established the following:

\begin{theorem} 
[{\cite[Theorems 1.9 \& 1.10]{blumberg2017universality}}]
\label{thm:stability-HI}
Let $X$ and $Y$ be totally bounded metric spaces. Then, for any $k\in \bbZ_+,$
	$$\di (\pHVR[\bullet]{k}{X},\pHVR[\bullet]{k}{Y})\leq \dhi(  \VR_{\bullet}(X)  ,  \VR_{\bullet}(Y)  )\leq  2\cdot \dgh(X,Y).$$
\end{theorem}

\subsection{Sullivan minimal models} \label{sec:smm background}

Sullivan minimal models have been important tools in rational homotopy theory, providing a simplified algebraic representation of topological spaces that facilitates the study of their properties and invariants.  
We review Sullivan minimal models from \cite{felix2001rational,Hess06,felix2017rational}.

\subsubsection{Commutative differential graded algebras (CGDAs)}

We work with graded vector spaces over \(\mathbb{Q}\), graded over the non-negative integers. A \emph{graded algebra} \(A\) is a graded vector space \(A = \bigoplus_{n\geq 0} A^n\) equipped with an associative multiplication \(A \otimes A \to A\) of degree \(0\), and a unit element \(1 \in A^0\).

\begin{definition}
\label{def:cdga}
A \emph{commutative differential graded algebra (CDGA)} is a graded algebra \(A = \bigoplus_{n\geq 0} A^n\) equipped with a differential \(\partial: A \to A\) satisfying
\[
\partial(ab) = (\partial a)b + (-1)^{\deg a} a(\partial b), \quad \partial^2 = 0,
\]
and the graded commutativity condition:
\[
ab = (-1)^{\deg a \cdot \deg b} ba,
\]
for all homogeneous elements \(a, b \in A\).  
If \(\partial\) has degree \(-1\) (resp. \(+1\)), then \((A, \partial)\) is called a \emph{chain} (resp. \emph{cochain}) algebra.

A CDGA \((A, \partial)\) is \emph{path-connected} if \(\rmH^0(A) = \bbQ\), and \emph{simply-connected} if, in addition, \(\rmH^1(A) = 0\).

\end{definition}

When the differential \(\partial\) is clear from context, we often omit it and refer to \(A\) as the CDGA.

\begin{example}\label{ex:Q as CDGA} 
The field \(\mathbb{Q}\) can be regarded as a CDGA concentrated in degree zero: let \(A^0 = \mathbb{Q}\), \(A^n = 0\) for all \(n > 0\), and define the differential \(\partial = 0\).
\end{example}

\begin{definition}
    A \emph{morphism of CDGAs} \(f: (A, \partial) \to (B, d)\) is a graded algebra homomorphism that preserves degrees, i.e., \(f(A^n) \subseteq B^n\) for all \(n\), and commutes with differentials, meaning \(f \circ \partial = d \circ f\).

    The morphism \(f\) is a \emph{quasi-isomorphism}, denoted \(f: A \xrightarrow{\simeq} B\), if it induces an isomorphism on cohomology.
    We say $A$ and $B$ are \emph{quasi-isomorphic} (via $f$).

    Two CDGAs \((A, \partial)\) and \((B, d)\) are \emph{weakly equivalent}, written \((A, \partial) \simeq (B, d)\), if there exists a finite zigzag of quasi-isomorphisms connecting them.
\end{definition}

Let $\CDGA$ denote the category of path-connected CDGAs.  

\medskip

A \emph{model structure} on a category \(\catC\) consists of three distinguished classes of morphisms—called \emph{weak equivalences}, \emph{fibrations}, and \emph{cofibrations}—satisfying a collection of axioms; see \cite[page 233]{Quillen1969} for details. A category equipped with such a structure is called a \emph{model category}. The associated \emph{homotopy category}, denoted \(\Ho(\catC)\), has as objects the bifibrant objects of \(\catC\) (i.e., both fibrant and cofibrant), and morphisms are homotopy classes of morphisms in \(\catC\). \label{para:homotopy category}

The category \(\CDGA\) admits a model structure in which:
\begin{itemize}
    \item the weak equivalences are the quasi-isomorphisms, 
    \item the fibrations are the surjective morphisms,
    \item cofibrantions include all Sullivan CDGAs, and in particular, all minimal Sullivan algebras are cofibrant.
\end{itemize}
See \cite[page 335]{gelfand2013methods} for more details.

\subsubsection{Sullivan algebras}

The \emph{free CDGA}, also called the \emph{exterior algebra}, generated by a graded vector space $V=\bigoplus_{n\geq 0} V^n$ is a commutative graded algebra defined as the quotient
\[\wedge V:=\left(\bigoplus_{k=0}^{\infty}\big(\otimes ^k V\big)\right)\Big/ \big \langle v_1\otimes v_2-(-1)^{\deg v_1\deg v_2} v_2\otimes v_1,\, \forall v_1, v_2\in V \big \rangle.\]
Note that $\wedge V = \bigoplus_{\ell>0}^\infty \wedge^\ell V$, where $\wedge^\ell V$ is the linear span of the elements \(v_1 \wedge \cdots \wedge v_\ell\), where each \(v_i \in V\); 
these elments have degree equal to $\sum_i \deg v_i$ and \emph{word-length} \(\ell\). 
We write \(\wedge^{\geq \ell} V := \bigoplus_{k \geq \ell} \wedge^k V\) and $\wedge^+ V := \wedge^{\geq 1}V$.

\begin{definition}[{\cite[page 181]{felix2001rational}}]  
\label{def:sullivan alg}  
    A \emph{Sullivan algebra} is a CDGA of the form \((\wedge V, \partial)\), where $V=\bigoplus_{n> 0} V^n$ and \(V = \bigcup_{k=0}^{\infty} V(k)\) is a filtered union of graded subspaces  
    \(
    V(0) \subset V(1) \subset \cdots
    \)  
    satisfying:
    \[
    \partial: V(0) \to 0 \quad \text{and} \quad \partial: V(k) \to \wedge V(k-1) \text{ for all } k \geq 1.
    \]  
    
    A Sullivan algebra \((\wedge V, \partial)\) is \emph{minimal} if  
    \(
    \im(\partial) \subset \wedge^{+} V \cdot \wedge^{+} V.
    \)  
\end{definition}  

Let $\wedge(t,\partial t)$ be the free CDGA on the basis $\{t, \partial t\}$ with $\deg t=0$ and $\deg \partial t=1$.
Let $\epsilon_0,\epsilon_1:\wedge(t,\partial t)\to \bbQ$ be two CDGA morphisms given by $\epsilon_0(t)=0$ and $\epsilon_1(t)=1$, respectively. 
\begin{definition}
\label{def:homotopy of CDGA morphism}
    Two CGDA morphisms $\varphi_0,\varphi_1:(\wedge V,\partial)\to (A,\partial)$ from a Sullivan algebra to a general CDGA are \emph{homotopic}, denoted by 
    \(\varphi_0 \sim \varphi_1,\) 
    if there is a CGDA morphism $\Phi:(\wedge V,\partial)\to \wedge(t,\partial t)\otimes(A,\partial)$ such that $(\epsilon_i\otimes\id)\circ\Phi =\varphi_i$. 
    Such $\Phi$ is called a \emph{homotpy} from $\varphi_0$ to $\varphi_1$.
\end{definition}



\subsubsection{Simplicial CDGAs}

A \emph{simplicial set} consists of a sequence of sets \(\{X_n\}_{n \geq 0}\), together with \emph{face maps} \(\partial_i: X_n \to X_{n-1}\) and \emph{degeneracy maps} \(s_i: X_n \to X_{n+1}\) for all \(i = 0, \dots, n\), satisfying the following \emph{simplicial identities}:
\begin{itemize}
    \item \(\partial_j \partial_i = \partial_i \partial_{j-1}\) for \(i < j\).
    \item \(s_j s_i = s_i s_{j+1}\) for \(i \leq j\).
    \item 
   \(
   \partial_j s_i =
   \begin{cases}
   s_{i-1} \partial_j & \text{if } j < i, \\
   \operatorname{id}  & \text{if } j = i \text{ or } j = i+1, \\
   s_i \partial_{j-1} & \text{if } j > i+1.
   \end{cases}
   \)
\end{itemize}
A \emph{simplicial morphism} is a map between simplicial sets \(\{X_n\}\) and \(\{Y_n\}\) is a collection of maps \(f_n: X_n \to Y_n\) that commute with face and degeneracy maps.
Let $\SSet$ be the category of simplicial sets.

A \emph{simplicial CDGA} is a simplicial object in the category of CDGAs, meaning it consists of a sequence of CDGAs with face and degeneracy maps that respect both the simplicial and differential structures.

\begin{definition}
The \emph{algebra of polynomial differential forms}, denoted by $\APL$, is the simplicial CGDA given as follows: for each $n\geq 0$,
\begin{equation}\label{eq:A_PL CDGA}
    (\APL)_n := \wedge(t_0,\cdots,t_n, \partial t_0,\cdots,\partial t_n)\Big/\left(\sum t_i-1,\sum \partial t_i\right) \in \CDGA,
\end{equation}
where $t_i$ denotes a variable of degree $0$ for each $0\leq i\leq n$ and $n\geq 0$. The face and degeneracy morphisms are given by
    \[\partial_i(t_j)=
    \begin{cases}
        t_j,&\mbox{$j<i$,}\\
        0,&\mbox{$j=i$,}\\
        t_{j-1},&\mbox{$j>i$,}
    \end{cases}
    \quad\text{ and }\quad
    s_j(t_i)=
    \begin{cases}
        t_i,&\mbox{$i< j,$}\\ 
        t_i+t_{i+1},&\mbox{$i=j$,}\\
        t_{i+1}.&\mbox{$i> j.$}
    \end{cases}
    \]
\end{definition}

\begin{remark}
    The name and notation \(\APL\) originate from the fact that each element in \((\APL)_n\) corresponds to a polynomial differential form on the Euclidean \(n\)-simplex, given by  
    \[
    \Delta^n := \left\{ (t_0, \dots, t_n) \in \mathbb{R}^{n+1} \mid t_0 + \cdots + t_n = 1, \quad t_i \geq 0 \text{ for all } i \right\}.
    \]
\end{remark}

We observe that, as a simplicial CDGA, \(\APL\) is characterized by two parameters. The first, typically denoted by the subscript \(n\), corresponds to the \emph{simplicial dimension}, while the second, often denoted by the superscript \(p\), represents the \emph{grading}. 
\begin{itemize}
    \item Fixing \(n\) while varying \(p\) yields a CDGA; see \cref{eq:A_PL CDGA}.
    \item Fixing \(p\) while varying \(n\) defines a simplicial set: 
    \begin{equation}\label{eq:A_PL SSet}
     \left( \APL\right)^p :=\left\{ \left( \APL\right)^p_n \right\}_{n\geq 0}\in \SSet.
    \end{equation}
\end{itemize}
   
Given $K\in \SSet$, let $\SSet\left( K,\left( \APL\right)^p\right)$ be the set of simplicial set morphisms from $K$ to $\left( \APL\right)^p$. Define
\[\SSet(K,\APL):= \bigoplus_{p\geq 0}\, \SSet\left( K,\left( \APL\right)^p\right), \]
which naturally admits a CDGA structure. In particular, we obtain a functor 
\[
\begin{tikzcd}[row sep=tiny, column sep = small]
    {\SSet(\cdot, \APL)}: &\SSet \arrow[r] & \CDGA \\
    & K \arrow[r, maps to] & \SSet(K, \APL).
\end{tikzcd}
\]

\begin{definition}[$\APL(X)$]
\label{def:A_PL(X)}
    For any topological space $X$, let $\rmS_{*}(X):=\{\rmS_n(X)\}_{n\geq 0}$ be the singular complex of $X$, where each $\rmS_n(X)$ is the set of singular simplices $\sigma:\Delta^n\to X$. The \emph{cochain algebra of polynomial forms on $X$} is defined as \[\APL(X):=\SSet(\rmS_{*}(X),\APL).\] 
\end{definition} 

\subsubsection{Sullivan minimal models}
 \label{subsec:S.m.m for space}

\begin{definition}[$\smm{A,\partial}$ and $\smm{X}$] 
\label{def:Smm of space}
A \emph{Sullivan model} for a CDGA $(A,\partial)$ is a Sullivan algebra, denoted by $\smm{A,\partial}: = (\wedge V,\partial)$, together with a quasi-isomorphism 
    \[\mu_A:(\wedge V,\partial)\xrightarrow{\simeq} (A,\partial).\]
If \((\wedge V, \partial)\) is minimal, then it is called the \emph{Sullivan minimal model} for $(A,\partial)$.
For simplicity, we often omit the quasi-isomorphism \(\mu_A\) and write \(\smm{(A, \partial)}\).

A \emph{Sullivan (minimal) model} for a path-connected topological space \(X\) is a Sullivan (minimal) model for \(\APL(X)\), denoted \(\smm{\APL(X)}\) or simply \(\smm{X}\).
\end{definition}


\begin{example}[{\cite[page 7]{felix2017rational}}]\label{ex:s.m.m} 
Let $n\geq 1$.
\begin{itemize}
    \item $\smm{\bbS^{2n-1}}=(\wedge u,0)$ with $\deg u=2n+1$;
    \item $\smm{\bbS^{2n}}=(\wedge (a,b),\partial)$ with $\deg a=2n$ and $\partial a=0,\partial b=a^2$, for $n\geq 1$;
    \item $\smm{X\times Y}\cong \smm{X}\otimes \smm{Y}$ if one of $\rmH^*(X;\bbQ)$ or $\rmH^*(Y,\bbQ)$ if of finite type;
    \item $\smm{X\vee Y}\simeq \smm{X}\oplus_{\bbQ} \smm{Y}$;
    \item $\smm{K(\bbZ,n)}=(\wedge a,0)$ with $\deg a=n$, where $K(\bbZ,n)$ is the Eilenberg–MacLane space such that its $n$-th homotopy group is isomorphic to $\bbZ$ and all other homotopy groups are trivial. 
\end{itemize}
\end{example}

\begin{proposition} [{\cite[\textsection 12]{felix2001rational}}]\label{prop:iso=q.iso} 
\label{prop:existence and uniqueness}
\begin{enumerate}
    \item A quasi-isomorphism between two simply-connected minimal Sullivan algebras is an isomorphism.
    \item Let $A$ be a path-connected CDGA (or let $X$ be a path-connected space). Then a Sullivan minimal model for $A$ (or $X$) exists and is unique up to isomorphism of CDGAs.  
\end{enumerate}
\end{proposition}

\subsubsection{Sullivan representative of a map}
For a map $f:(A,\partial)\to (B,\partial)$ between two path-connected CDGAs, there is an induced map $\smm{f}$ between the Sullivan minimal models of $(A,\partial)$ and $(B,\partial)$, given by the following proposition.
\begin{proposition}[{\cite[Proposition 12.9]{felix2001rational}}]
\label{prop: S-representative} 
Let $f:(A,\partial)\to (B,\partial)$ be a CDGA morphism between path-connected CDGAs.  
Let $\mu_A:(\wedge V,\partial)\xrightarrow{\simeq}A$ and $\mu_B:(\wedge W,\partial)\xrightarrow{\simeq}B$ be Sullivan minimal models of $A$ and $B$, respectively. 
Then there is a morphism $\smm{f}:(\wedge V,\partial)\to (\wedge W,\partial)$ such that
\begin{center}
	\begin{tikzcd}
(\wedge V,\partial)
	\ar[d, "\smm{f}" left]
	\ar[r, "\mu_A" above, "\simeq" below] 
    & 
A
	\ar[d,"f" right]
		\\
(\wedge W,\partial)
	\ar[r,"\mu_B" below, "\simeq" above]
	& 
B.
	\end{tikzcd}
\end{center}
commutes up to homotopy. 
The morphism $\smm{f}$ is called a \emph{Sullivan representative} of $f.$ The homotopy class of $\smm{f}$ is determined by the homotopy class of $f$.
\end{proposition}

\begin{remark}\label{rmk:SAlg functor}
\cref{prop:iso=q.iso} and \cref{prop: S-representative} together implies that the map assigning to a path-connected CDGA $(A,\partial)$ a Sullivan minimal model $\smm{A,\partial}$ and to a CDGA morphism $f$ a Sullivan representative $\smm{f}$ gives rise to a functor
\[\smm{\cdot}:\Ho(\CDGA)\to \Ho(\CDGA) .\]
Here, $\Ho(\CDGA)$ denotes the homotopy category of $\CDGA$ (see page \pageref{para:homotopy category} for details). 
\end{remark}

\subsubsection{Relation to rational homotopy groups and rational (co)homology}

Sullivan minimal models encode rich homotopical information. In particular, they recover both the rational homotopy groups and the rational (co)homology of simply-connected spaces.

Given two graded vector spaces $V=\bigoplus_{n\geq 0} V^n$ and $W=\bigoplus_{n\geq 0} W^n$, a linear map $f:V\to W$ is of \emph{degree $k$} if
if it satisfies  
\[
f(V^n) \subseteq W^{n+k} \quad \text{for all } n \in \mathbb{Z}.
\]  
Equivalently, \( f \) is a family of linear maps \( f_n: V^n \to W^{n+k} \) for each \( n \). 
We denote by 
\[\Hom^k_{\GVect}(V,W):=\prod_{n\geq 0} \Hom_{\vs}(V^n,W^{n+k})\] 
the set of linear maps of degree $k$ from $V$ to $W$, and by $\Hom_{\GVect}(V,W)$ the graded vector space whose elements of degree $k$ are the linear maps from $V$ to $W$ of degree $k$.
That is,
\[\Hom_{\GVect}(V,W):=\bigoplus_{k\in \bbZ}\Hom^k(V,W) = \bigoplus_{k\in \bbZ} \prod_{n\geq 0} \Hom_{\vs}(V^n,W^{n+k}).\]

For example, if $W=\bbQ$ (cf. \cref{ex:Q as CDGA}), then $\Hom^k_{\GVect}(V,\bbQ)\cong V^{-k}$ for all $k\in \bbZ$ and  
\[\Hom_{\GVect}(V,\bbQ):=\bigoplus_{k\in \bbZ}V^{-k}.\]

Below is a classical result relating rational homotopy groups to Sullivan models.

\begin{theorem} [{\cite[Theorem 4.1]{felix2017rational}}] \label{thm:rational homotopy v.s. V}
If $\mu_X:(\wedge V,\partial)\xrightarrow{\simeq}\APL(X)$ is a Sullivan minimal model of a simply-connected space with finite Betti numbers, then there is an isomorphism of graded vector spaces:
\[\pi_*(\mu_X): \pi_*(X)\otimes\bbQ \xrightarrow{\cong} \Hom_{\GVect}(V,\bbQ),\]
where $\pi_*(\mu_X)$ is natural with respect to morphisms of minimal Sullivan algebras and with respect to based maps of CW complexes.
\end{theorem}

Rational cohomology is also recovered from the minimal model via
\[
\rmH^*(\wedge V, \partial) \cong \rmH^*(X; \mathbb{Q}).
\]


Finally, the following rational Hurewicz theorem connects rational homotopy and homology.
It is first due to Dyer \cite{dyer1972}, and a more elementary proof is later established in \cite{klaus_kreck_2004}. 

\begin{theorem}[Rational Hurewicz, {\cite{dyer1972,klaus_kreck_2004}}]
Let \(X\) be a simply-connected space such that \(\pi_i(X) \otimes \mathbb{Q} = 0\) for \(1 < i < n\). Then the rational Hurewicz map
\[
h \colon \pi_i(X) \otimes \mathbb{Q} \to \rmH_i(X; \mathbb{Q})
\]
is an isomorphism for \(1 \leq i < 2n - 1\) and a surjection for \(i = 2n - 1\).
\end{theorem}


\section{Persistent Sullivan Minimal Models}
\label{sec:pSMM}

In this section, we define the \emph{persistent} Sullivan minimal models of persistent topological spaces and prove some stability results for the interleaving distance (in the homotopy category of CDGAs) between persistent Sullivan minimal models. 

\subsection{Persistent Sullivan minimal models} 

Recall from \cref{sec:smm background} that the Sullivan minimal model of a topological space \(X\) is defined as a minimal Sullivan algebra \(\smm{X}\), together with a quasi-isomorphism \(\mu_X: \smm{X} \to \APL(X)\), where \(\APL(X)\) denotes the CDGA of polynomial differential forms on \(X\).  
We extend this notion to persistent spaces.

Before giving the definition, we note that the homotopy category of commutative differential graded algebras is naturally contravariant with respect to topological spaces. 
That is, a continuous map \( f: X \to Y \) induces a CDGA morphism \( \APL(f): \APL(Y) \to \APL(X) \). 
To properly capture this contravariance, we use the opposite category construction: given a category \(\catC\), its \emph{opposite category} \(\catC^{\mathrm{op}}\) has the same objects as \(\catC\), but with all morphisms reversed.

\begin{definition} \label{def:psmm}
    Let $\Xfunc:(\bbR,\leq )\to  \topo$ be a persistent path-connected space. We define the \emph{persistent Sullivan minimal model} of $\Xfunc$ to be a persistent minimal Sullivan algebra $\smm{X_\bullet}$ together with CDGA quasi-isomorphisms $\mu_{X_\bullet}:=\left\{\mu_{X_t}:\smm{X_t}\xrightarrow{\simeq} \APL(X_t)\right\}_t$ such that: 
\begin{itemize}
    \item For each $t,$ $\smm{X_t}$ together with $\mu_{X_t}$ is a Sullivan minimal model for $X_t$; and
    \item For any $t\leq s$, the following diagram commutes up to homotopy:
    \[
    \begin{tikzcd}
        \smm{X_t}
        \ar[r,"\mu_{X_t}" above, "\simeq" below]
        &
        \APL(X_t)
        \\
        \smm{X_s}
         \ar[u,"\smm{f_{t,s}}" left]
        \ar[r,"\mu_{X_s}" below, "\simeq" above]
        &
        \APL(X_s).
        \ar[u, "f_{t,s}" right]
    \end{tikzcd}
    \]
    In particular, \(\mu_{X_\bullet}\) induces a natural isomorphism between the functors  
    \[
    \mathrm{Ho} \circ \mathfrak{M} \circ \APL \circ \Xfunc  \xRightarrow[\cong]{\mu_{X_\bullet}} \mathrm{Ho} \circ \APL \circ \Xfunc :\, (\mathbb{R},\leq) \to \Ho(\CDGA)^{\op}.
    \]  
\end{itemize} 
We often omit explicit mention of \(\mu_{X_\bullet}\) and refer to $\smm{X_\bullet}$  as a persistent Sullivan minimal model for $X_\bullet$. 
\end{definition}

The existence and uniqueness (up to quasi-isomorphism in $\CDGA$) of persistent Sullivan minimal models of persistent path-connected spaces follow from \cref{prop:existence and uniqueness} and \cref{prop: S-representative}.

If the persistent topological space is given by the Vietoris-Rips filtration of a metric space $X$, then we denote the resulting persistent Sullivan minimal model as $\smm{\VR_\bullet(X)}$.

\begin{example}
Let $\bbS^1$ be the unit geodesic circle.
It is shown in \cite[Theorem 7.6]{adamaszek2017vietoris} that there are natural homotopy equivalences
 \begin{equation*}
 \VR_{ t}(\bbS^1) \simeq \begin{cases} 
 \bbS^{2l+1},&\mbox{if $\tfrac{2l}{2l+1}\pi<t\leq\tfrac{2l+2}{2l+3}\pi$ for some $l=0,1,\cdots$,}\\
 \ast, &\mbox{if $t> \pi.$}
 \end{cases}
 \end{equation*}

Combined with \cref{ex:s.m.m}, we see that a persistent Sullivan minimal model for the Vietoris-Rips filtration of $\bbS^1$ is 
 \begin{equation*}
\smm{\VR_\bullet(\bbS^1)} =\bigoplus_{l=0}^{+\infty}(\wedge u_{2l+1},0)\left( \tfrac{2l}{2l+1}\pi,\tfrac{2l+2}{2l+3}\pi\right],
 \end{equation*}
 where each $(\wedge u_{2l+1},0)\left( \tfrac{2l}{2l+1}\pi, \tfrac{2l+2}{2l+3}\pi\right]$ is an interval-like persistent Sullivan algebra (see \cref{def:interval-like persistent object}).
\end{example}

\subsection{Stability of persistent Sullivan minimal models} 
\label{sub:stability}
We prove \cref{thm:stab-PM}.

\begin{lemma}\label{lem:diagram} Let $\smm{X_1},\smm{X_2},\smm{Y_1}$ and $\smm{Y_2}$ be Sullivan minimal models of path-connected topological spaces $X_1,X_2,Y_1$ and $Y_2$, respectively. Then, a commutative diagram of topological spaces
	\begin{center}
		\begin{tikzcd}
		X_1 \ar[d, "f" left] &
		X_2 \ar[l,"i" above]
		\ar[d,"g"]
			\\
		 Y_1
			& Y_2 \ar[l,"j" below]
		\end{tikzcd}
	\end{center} 
induces a diagram of the Sullivan minimal models, which commutes in $\Ho(\CDGA)$:
	\begin{center}
		\begin{tikzcd}
		\smm{X_1}  
		\ar[r,"\smm{i}" above]
		&
		\smm{X_2}
			\\
		 \smm{Y_1} \ar[u, "\smm{f}" left]
			\ar[r,"\smm{j}" below]
			& \smm{Y_2}. 
		\ar[u,"\smm{g}" right]
		\end{tikzcd}
	\end{center}
\end{lemma}

\begin{proof} We first apply the contravariant functor $\APL$ to the diagram of topological spaces and obtain the middle square of the following diagram:
	\begin{center}
		\begin{tikzcd}
		\textcolor{red}{\smm{X_1}}
		\ar[rrr,red, "\smm{i}" above]
		\ar[rd,blue, "\simeq" above]
		&&&
		\textcolor{red}{\smm{X_2}}
		\ar[ld,blue, "\simeq" above]
		\\
		&
		\APL(X_1) 
		\ar[r,"\APL(i)" above]
		&
		\APL(X_2)
		&
			\\
			&
		 \APL(Y_1) \ar[u, "\APL(f)" left]
			\ar[r,"\APL(j)" below]
			& \APL(Y_2)
		\ar[u,"\APL(g)" right]
		&
		\\
		\textcolor{red}{\smm{Y_1}}
		\ar[uuu,red, "\smm{f}" left]
			\ar[rrr,red,"\smm{j}" below]
			\ar[ur,blue, "\simeq" below]
		&&&
		\textcolor{red}{\smm{Y_2}}
		\ar[uuu,red,"\smm{g}" right]
		\ar[lu,blue, "\simeq" below]
		\end{tikzcd}
	\end{center}
The blue arrows are quasi-isomorphisms obtained from the existence of Sullivan minimal models, cf. \cref{prop:existence and uniqueness}. \cref{prop: S-representative} gives us the red arrows as the Sullivan representatives of the corresponding black arrows. In addition, each one of the four side squares, which are in a trapezoidal shape, commutes up to homotopy and thus commutes in the category $\Ho(\CDGA)$. Because the quasi-isomorphism blue arrows can be reversed in $\Ho(\CDGA)$, we can conclude that the outer (red) square commutes.
\end{proof}	

The following corollary follows directly from \cref{lem:diagram}:
\begin{corollary} \label{cor:di-stab}
Let $\Xfunc$ and $\Yfunc$ be persistent spaces with persistent Sullivan minimal models $\smm{\Xfunc}$ and $\smm{\Yfunc}$, respectively. Then 
\[\di^{\Ho(\CDGA)}(\smm{\Xfunc}, \smm{\Yfunc}) \leq \di^{\topo}(\Xfunc,\Yfunc).\]
\end{corollary}

\begin{lemma}\label{lem:X simeq Y} Let $\Xfunc\simeq \Yfunc$ be weakly equivalent persistent spaces (see page \pageref{para:weakly equivalent}). 
Assume that $\smm{\Xfunc}$ and $\smm{\Yfunc}$ are persistent Sullivan minimal models of $\Xfunc$ and $\Yfunc$, respectively. Then $\smm{\Xfunc}\cong \smm{\Yfunc}$.
\end{lemma}

\begin{proof}Since $\Xfunc\simeq \Yfunc$, there exist an persistent space $\Ufunc$ and natural transformations $f:\Ufunc\Rightarrow \Xfunc$ and $g:\Ufunc\Rightarrow \Yfunc$ that are (objectwise) weak equivalences. For any $t$, $f_t$ is a weak equivalence. 
Thus, the Sullivan representative of $f_t$ is a quasi-isomorphism from the Sullivan minimal model for $X_t$ to the Sullivan minimal model for $U_t.$ By applying \cref{lem:diagram}, we decude that $\smm{\Xfunc}$ and $\smm{\Ufunc}$ are isomorphic to each other. Similarly, we have $\smm{\Yfunc}\cong \smm{\Ufunc}$.
\end{proof}	

\begin{proof}[Proof of \cref{thm:stab-PM}] Let $\delta>\dhi(\Xfunc,\Yfunc)$. Then there exist persistent spaces $\Xfunc'\simeq \Xfunc$ and $\Yfunc'\simeq \Yfunc$ such that $\Xfunc'$ and $\Yfunc'$ are $\delta$-interleaved. By \cref{cor:di-stab} and \cref{lem:X simeq Y}, we have
\begin{align*}
    &\di^{\Ho(\CDGA)}(\smm{\Xfunc}, \smm{\Yfunc})\\
    \leq\,& \di^{\Ho(\CDGA)}(\smm{\Xfunc}, \smm{\Xfunc'})+\di^{\Ho(\CDGA)}(\smm{\Xfunc'}, \smm{\Yfunc'})+\di^{\Ho(\CDGA)}(\smm{\Yfunc'}, \smm{\Yfunc})\\
    \leq\,& 0+\di(\Xfunc',\Yfunc')+0=\delta.
\end{align*}
Since this holds for any $\delta > \dhi(\Xfunc, \Yfunc)$, it follows that
\[
    \di^{\Ho(\CDGA)}(\smm{\Xfunc}, \smm{\Yfunc}) \leq \dhi(\Xfunc, \Yfunc).
\]

For the case of Vietoris-Rips filtrations of metric spaces, the result follows from the Gromov-Hausdorff stability of the homotopy interleaving distance between Vietoris-Rips filtrations (see \cref{thm:stability-HI}).
\end{proof}

\subsection{Lower bounds for \texorpdfstring{$\di^{\Ho(\CDGA)}$}{dicdga} between persistent Sullivan minimal models}

In \cref{sec:linear parts}, we recall from \cite[\textsection 12(b)]{felix2001rational} the functor $\rmQ:\CDGA\to \GVect$, which extracts the linear part (or indecomposables) of a CDGA. 
Moreover, we recall that if two minimal Sullivan algebras \((\wedge V, \partial)\) and \((\wedge W, \partial)\) are quasi-isomorphic, then:
\begin{enumerate}
    \item \(V \cong W\) as graded vector spaces (see \cref{cor:V cong W}), and
    \item \(\rmH^*(\wedge V, \partial) \cong \rmH^*(\wedge W, \partial)\) as graded vector spaces.
\end{enumerate}
Inspired by these facts, we establish two lower bounds for the interleaving distance (in the homotopy category of CDGAs) between persistent minimal Sullivan algebras:

\lowerBound*

We denote by $\pcHVR[\bullet]{*}{X}$ the persistent cohomology of the Vietoris-Rips filtration of a metric space $X$, and denote by $\Vfunc^X$ the underlying graded vector space of a persistent Sullivan minimal model of  $\VR_\bullet(X)$. 
The following results follow immediately from \cref{thm:lower-bound}.

\begin{corollary} \label{cor:lower bound}
    Let $X$ and $Y$ be totally bounded metric spaces. Then,
    \begin{equation}\label{eq:di-cohomology}
        \di^{\GVect} \left( \pcHVR[\bullet]{*}{X},\pcHVR[\bullet]{*}{Y}\right)\leq \di^{\Ho(\CDGA)}(\smm{\VR_\bullet(X)}, \smm{\VR_\bullet(Y)}).
    \end{equation}
    If $X$ and $Y$ are simply-connected, then 
    \begin{equation}\label{eq:di-V-W} 
        \di^{\GVect}\left( \pi_*(\VR_\bullet(X))\otimes \bbQ, \pi_*(\VR_\bullet(Y))\otimes \bbQ \right) \leq \di^{\Ho(\CDGA)}(\smm{\VR_\bullet(X)}, \smm{\VR_\bullet(Y)}).
    \end{equation}
\end{corollary}

Moreover, the left-hand side of \cref{eq:di-cohomology} is no greater than that of \cref{eq:di-V-W}, as shown in \cite{memoli2024persistent}. See also \cite[Example~6.4]{memoli2024persistent} for an instance in which the (co)homological lower bound is strictly smaller than the rational homotopy lower bound.

\subsubsection{Linear parts of CDGA}\label{sec:linear parts} 

Let \( A^+ := \bigoplus_{n > 0} A^n \) denote the positive-degree part of a CDGA \( (A, \partial) \). The \emph{linear part} (or \emph{indecomposables}) of \( A \) is defined as
\[
\rmQ(A) := A^+ / (A^+ \cdot A^+).
\]
For \( a \in A^+ \), we denote by \( \rmQ(a) \) the image of \( a \) under the canonical projection \( A^+ \twoheadrightarrow \rmQ(A) \). 

Since \( \partial(A^+ \cdot A^+) \subset A^+ \cdot A^+ \) by the Leibniz rule, the differential on \( A \) induces a well-defined linear map
\[
\rmQ(\partial) : \rmQ(A) \to \rmQ(A), \quad \rmQ(a) \mapsto \rmQ(\partial a),
\]
In addition, \( \rmQ(\partial)^2 = \rmQ(\partial^2) = 0.\)
This gives \( (\rmQ(A), \rmQ(\partial)) \) the structure of a differential graded vector space.
More generally, a \emph{differential graded vector space} is a graded vector space \(V = \bigoplus_n V^n\) equipped with a linear differential \(d: V \to V\) of degree \(+1\) such that \(d^2 = 0\). 
We write \(\dgvs\) for the category of differential graded vector spaces.

Given a morphism \( \varphi : (A, \partial) \to (B, \partial) \) of CDGAs, define the induced map on indecomposables by
\[
\rmQ(\varphi) : \rmQ(A) \to \rmQ(B), \quad \rmQ(a) \mapsto \rmQ(\varphi(a)).
\]
This is well-defined since \( \varphi(A^+ \cdot A^+) \subset B^+ \cdot B^+ \), and satisfies
\[
\rmQ(\partial) \circ \rmQ(\varphi) = \rmQ(\varphi) \circ \rmQ(\partial),
\]
so \( \rmQ(\varphi) \) is a morphism of differential graded vector spaces.

\begin{lemma} \label{fact:Q is functor}
The assignment \( \rmQ: \CDGA \to \dgvs \), which maps each CDGA to its linear part and each morphism to its induced map on linear parts, defines a functor.
\end{lemma}

\begin{proof}
    We have seen that \((\rmQ(A), \rmQ(\partial))\) is an object in \(\dgvs\). It remains to verify functoriality.
    
    For the identity morphism \(\id_A: A \to A\), we have \(\rmQ(\id_A) = \id_{\rmQ(A)}\) by definition.
    For a composition of CDGA morphisms \((A, \partial) \xrightarrow{\varphi} (B, \partial) \xrightarrow{\psi} (C, \partial)\), we note that \(\varphi(A^+ \cdot A^+) \subset B^+ \cdot B^+\) and \(\psi(B^+ \cdot B^+) \subset C^+ \cdot C^+\), so \((\psi \circ \varphi)(A^+ \cdot A^+) \subset C^+ \cdot C^+\), ensuring that \(\rmQ(\psi \circ \varphi)\) is well-defined. For any \(a \in A^+\), we have
    \[
    \rmQ(\psi \circ \varphi)(a) = \rmQ(\psi(\varphi(a))) = \rmQ(\psi)(\rmQ(\varphi)(a)),
    \]
    hence \(\rmQ(\psi \circ \varphi) = \rmQ(\psi) \circ \rmQ(\varphi)\).
    
    Therefore, \(\rmQ\) defines a functor from \(\CDGA\) to \(\dgvs\).
\end{proof}

\begin{lemma} \label{lemma:Q-wedge-natural-iso}
Let \(V\) be a graded vector space concentreated in positive degrees.
The inclusion \(V \hookrightarrow \wedge V\) induces a canonical isomorphism of graded vector spaces
\(
\eta_V : V \xrightarrow{\cong} \rmQ(\wedge V),
\)
sending each \(v \in V\) to its class in the linear part \(\rmQ(\wedge V)\). 
These isomorphisms induces a natural isomorphism of functors
\[
\eta: \mathrm{id}_{\GVect} \Rightarrow \rmQ \circ \wedge(-), \quad \text{ with }V\mapsto \eta_V,\, f \mapsto \rmQ(\wedge f),
\]
for any $V\in \GVect$ and any morphism \(f: V \to W\) in \(\GVect\).
\end{lemma}

\begin{proof}
The map \(\eta_V\) is injective because for any distinct $v,v'\in V$, the difference $v-v'$ cannot be expressed as a product of positive-degree elements. 
It is surjective because every indecomposable element of \(\wedge V\) is represented by a generator. 

To verify naturality, let \(f: V \to W\) be a morphism of graded vector spaces. The induced CDGA map \(\wedge f: \wedge V \to \wedge W\) makes the following diagram commute:
\[
\begin{tikzcd}
V \arrow[r, "\eta_V"] \arrow[d, "f"'] & \rmQ(\wedge V) \arrow[d, "\rmQ(\wedge f)"] \\
W \arrow[r, "\eta_W"] & \rmQ(\wedge W)
\end{tikzcd}
\]
This shows that the family \(\{\eta_V\}\) defines a natural transformation \(\eta: \mathrm{id} \Rightarrow \rmQ \circ \wedge(-)\), and each \(\eta_V\) is an isomorphism.
\end{proof}


Recall from \cref{def:sullivan alg} that a Sullivan algebra \((\wedge V, \partial)\) is minimal if and only if \(\partial(V) \subset \wedge^{\geq 2}V\), or equivalently, \(\rmQ(\partial) = 0\). This immediately implies Item~\ref{prop:linear part-minimal} below. For Items~\ref{prop:linear part-homotopy} and~\ref{prop:linear part-quasi}, see \cite[Propositions 12.8 and 14.13]{felix2001rational}.

\begin{proposition} \label{prop:linear part}
Let \((\wedge V, \partial)\) and \((\wedge W, \partial)\) be Sullivan algebras. Then:
\begin{enumerate}
    \item \label{prop:linear part-minimal}
    \((\wedge V, \partial)\) is minimal if and only if \(\rmQ(\partial) = 0\).

    \item \label{prop:linear part-homotopy}
    If \(\varphi_0 \sim \varphi_1 : (\wedge V, \partial) \to (\wedge W, \partial)\) are homotopic CDGA morphisms between minimal Sullivan algebras and \(\rmH^1(\wedge V, \partial) = 0\), then \(\rmQ(\varphi_0) = \rmQ(\varphi_1)\).

    \item \label{prop:linear part-quasi}
    A morphism \(\varphi : (\wedge V, \partial) \to (\wedge W, \partial)\) is a quasi-isomorphism if and only if \(\rmH^*(\rmQ(\varphi))\) is an isomorphism.
\end{enumerate}
\end{proposition}

We deduce the following corollary from the above proposition.

\begin{corollary}\label{cor:V cong W}
If two minimal Sullivan algebras \((\wedge V, \partial_V)\) and \((\wedge W, \partial_W)\) are quasi-isomorphic, then \(V \cong W\) as graded vector spaces.
\end{corollary}

\begin{proof}
Since both algebras are minimal, \cref{prop:linear part} (\ref{prop:linear part-minimal}) implies that \(\rmQ(\partial_V) = \rmQ(\partial_W) = 0\), so the linear part carry zero differential. 
In addition, \cref{lemma:Q-wedge-natural-iso} implies 
\(
(\rmQ(\wedge V), 0) \cong V \quad \text{and} \quad (\rmQ(\wedge W), 0) \cong W
\)
as graded vector spaces, via the canonical projections \(\eta_V\) and \(\eta_W\).

Now let \(\varphi : (\wedge V, \partial_V) \to (\wedge W, \partial_W)\) be a quasi-isomorphism. By \cref{prop:linear part} (\ref{prop:linear part-quasi}), the induced map \(\rmH^*(\rmQ(\varphi))\) is an isomorphism. Since the differential on both sides is zero, this implies that \(\rmQ(\varphi)\) itself is an isomorphism of graded vector spaces. 
In summary, we have the commutative diagram below:
\begin{center}
    \begin{tikzcd}[column sep=large]
        \rmH^*(\rmQ(\wedge V), 0) 
        \ar[r, "\rmH^*(\rmQ(\varphi))", "\cong" below]
        & 
        \rmH^*(\rmQ(\wedge W), 0)\\
        \rmQ(\wedge V)
        \ar[r, "\rmQ(\varphi)", "\cong" below]
        \ar[u, "\cong" left]
        & \rmQ(\wedge W)
        \ar[u, "\cong" right]\\
        V 
        \ar[u, "\eta_V" right, "\cong" left]
        \ar[r, dashed, "\eta_W^{-1} \circ \rmQ(\varphi) \circ \eta_V"]
        & W
        \ar[u, "\eta_W" left, "\cong" right]
    \end{tikzcd}
\end{center}
Composing the isomorphisms along the bottom square yields the desired isomorphism \(V \cong W\).
\end{proof}

\begin{remark} Apparently, the inverse of \cref{cor:V cong W} is not true. Suppose $V=W$ has basis $\{a_2,b_3\}$, where $a_2$ and $b_3$ are elements of degree $2$ and $3$, respectively. Let $\partial_V:a_2\mapsto b_3$ and $\partial_V:b_3\mapsto 0$, and let $\partial_W=0$. Then $(\wedge V,\partial_V)$ and $(\wedge W,\partial_W)$ are not isomorphic, because $\rmH^2(\wedge V,\partial_V)=0\neq \bbQ\cong \rmH^2(\wedge V,\partial_W)$.
\end{remark}

\subsubsection{Lower bounds for \texorpdfstring{$\di^{\Ho(\CDGA)}$}{dicdga}}
\label{subsubsec:lower bounds}


We now prove \cref{thm:lower-bound} and \cref{cor:lower bound}.

\begin{proof}[Proof of \cref{thm:lower-bound}] 
Part (\ref{thm:interleaving-H}) follows from \cref{thm:Bubenik}. 
For Part (\ref{thm:interleaving-V,W}), we have
\begin{align*}
    \di^{\GVect}(\Vfunc,\Wfunc) 
    \, = \, & \di^{\GVect}(\rmQ\circ \wedge \Vfunc, \rmQ\circ \wedge \Wfunc) \quad \text{(by \cref{lemma:Q-wedge-natural-iso})}\\
    \, \leq \, & \di^{\Ho(\CDGA)}(\wedge\Vfunc, \wedge \Wfunc)  \quad \text{(by \cref{thm:Bubenik})}\\
    \, \leq \, & \di^{\CDGA}(\wedge\Vfunc, \wedge \Wfunc). \qedhere
\end{align*}

It is worth noting that the the inequalities in \cref{thm:lower-bound} do not always attain equality.
\cref{eq:stab 1}
can be strict when considering CDGAs with identical cohomology but differing Massey products. 
\cref{eq:stab 2} may be an equality in some special cases. 
For example, if the persistent CDGAs are constant, i.e., \(V_t = V\) and \(W_t = W\) for all \(t\), and \(\wedge V\) and \(\wedge W\) are quasi-isomorphic, then \cref{cor:V cong W} implies that \(V \cong W\), and the inequality becomes an equality. 
But, in general, \cref{eq:stab 2} can be strict, as demonstrated in \cref{ex:strict V v.s. wedge V}.

\begin{example}\label{ex:strict V v.s. wedge V}
    Consider two constant persistent spaces \(X_\bullet\) and \(Y_\bullet\) defined by \(X_t \simeq \mathbb{S}^2\) and \(Y_t \simeq K(\mathbb{Z},2) \times K(\mathbb{Z},3)\) for all \(t\). 
    By \cref{ex:s.m.m}, their Sullivan minimal models are \(\smm{\mathbb{S}^2} = (\wedge(a, b), \partial)\), with \(\deg a = 2\), \(\partial a = 0\), and \(\partial b = a^2\), and \(\smm{K(\mathbb{Z},2) \times K(\mathbb{Z},3)} = (\wedge(c, d), 0)\), with \(\deg c = 2\), \(\deg d = 3\).
    
    Define $V_t$ to be the graded vector space spanned by $a$ and $b$, and $W_t$ to be the graded vector space spanned by $c$ and $d$, for all \(t\). Then \(V_\bullet\) and \(W_\bullet\) are isomorphic as persistent graded vector spaces, so \(\di^{\GVect}(V_\bullet, W_\bullet) = 0\). 
    However, the CDGAs \((\wedge(a, b), \partial)\) and \((\wedge(c, d), 0)\) are not quasi-isomorphic, due to their differing differential structures. 
    Therefore, \(\di^{\Ho(\CDGA)}(\wedge V_\bullet, \wedge W_\bullet) = \infty\), showing that the inequality in \cref{eq:stab 2} can be strict.
\end{example}

These examples illustrate that while cohomology and linear parts provide meaningful lower bounds, they may fail to capture all the information present in persistent CDGA structures.
\end{proof}

\begin{proof}[Proof of \cref{cor:lower bound}] 
\cref{eq:di-cohomology} follows directly from \cref{thm:lower-bound} (\ref{thm:interleaving-H}). 

For \cref{eq:di-V-W}, assume that \(\smm{\VR_\bullet(X)} = (\wedge V_\bullet, \partial)\) and \(\smm{\VR_\bullet(Y)} = (\wedge W_\bullet, \partial)\), each equipped with quasi-isomorphisms
\[
\left\{\smm{\VR_t(X)} \xrightarrow{\mu_{\VR_t(X)}} \APL(\VR_t(X)) \right\}_{t >0}, \quad
\left\{\smm{\VR_t(Y)} \xrightarrow{\mu_{\VR_t(Y)}} \APL(\VR_t(Y)) \right\}_{t >0}.
\] 
By \cite[Lemma 3.1.7]{wilkins2011discrete} or \cite[Example 5.1]{memoli2024persistent}, since \(X\) and \(Y\) are simply-connected, each \(\VR_t(X)\) and \(\VR_t(Y)\) is simply-connected, and so are the Sullivan models \((\wedge V_t, \partial)\) and \((\wedge W_t, \partial)\) for all \(t > 0\).
Then applying \cref{thm:lower-bound} (\ref{thm:interleaving-V,W}), we obtain
\[
\di^{\GVect}(V_\bullet, W_\bullet) 
\leq 
\di^{\Ho(\CDGA)}\big(\smm{\VR_\bullet(X)}, \smm{\VR_\bullet(Y)}\big). \qedhere
\]

It follows from \cref{thm:rational homotopy v.s. V} that 
\[
\pi_*(\VR_\bullet(X)) \otimes \bbQ \cong \Hom_{\GVect}(V_\bullet, \bbQ)
\quad \text{and} \quad 
\pi_*(\VR_\bullet(Y)) \otimes \bbQ \cong \Hom_{\GVect}(W_\bullet, \bbQ).
\]
Therefore,
\begin{align*}
    \di^{\GVect}\big( \pi_*(\VR_\bullet(X)) \otimes \bbQ,\, \pi_*(\VR_\bullet(Y)) \otimes \bbQ \big)
    &= \di^{\GVect}\big( \Hom(V_\bullet, \bbQ),\, \Hom(W_\bullet, \bbQ) \big) \\
    &= \di^{\GVect}(V_\bullet, W_\bullet) \\
    &\leq \di^{\Ho(\CDGA)}\big( \smm{\VR_\bullet(X)},\, \smm{\VR_\bullet(Y)} \big).
\end{align*}
This completes the proof.
\end{proof}

\printbibliography



\end{document}